\documentclass[11pt, a4paper]{amsart}
\usepackage{amsmath}
\usepackage{amssymb}
\usepackage{color}
\usepackage{graphicx}

\newtheorem{teo}{Theorem}[section]

\theoremstyle{definition}

\def\ep{\varepsilon}

\def\a{\alpha}

\def\R{\mathbb R}

\begin{document}
\title[Asymptotics for inhomogenous equations with memory. Large dimensions]{Decay/growth rates for inhomogeneous heat equations with memory. The case of large dimensions}

\author[C. Cort\'{a}zar,  F. Quir\'{o}s \and N. Wolanski]{Carmen Cort\'{a}zar,  Fernando  Quir\'{o}s, \and Noem\'{\i} Wolanski}

\address{Carmen Cort\'{a}zar\hfill\break\indent
Departamento  de Matem\'{a}tica, Pontificia Universidad Cat\'{o}lica
de Chile \hfill\break\indent Santiago, Chile.} \email{{\tt
ccortaza@mat.puc.cl} }

\address{Fernando Quir\'{o}s\hfill\break\indent
Departamento  de Matem\'{a}ticas, Universidad Aut\'{o}noma de Madrid,
\hfill\break\indent 28049-Madrid, Spain,
\hfill\break\indent and Instituto de Ciencias Matem\'aticas ICMAT (CSIC-UAM-UCM-UC3M),
\hfill\break\indent 28049-Madrid, Spain.} \email{{\tt
fernando.quiros@uam.es} }

\address{Noem\'{\i} Wolanski \hfill\break\indent
IMAS-UBA-CONICET, \hfill\break\indent Ciudad Universitaria, Pab. I,\hfill\break\indent
(1428) Buenos Aires, Argentina.} \email{{\tt wolanski@dm.uba.ar} }

\keywords{Heat equation with nonlocal time derivative, Caputo derivative, asymptotic behavior.}

\subjclass[2010]{%
35B40, % Asymptotic behavior of solutions
35R11, % Fractional partial differential equations
35R09, %Integro-partial differential equations
45K05. %Integro-partial differential equations
}

\date{}

\begin{abstract}
We study the decay/growth rates in all  $L^p$ norms of solutions to  an inhomogeneous nonlocal heat equation in $\mathbb{R}^N$ involving a Caputo  $\alpha$-time derivative and a power $\beta$ of the Laplacian when the dimension is large, $N> 4\beta$. Rates depend strongly on the space-time scale and on the time behavior of the spatial $L^1$ norm  of the forcing term.
\end{abstract}

\maketitle

\begin{center}
	\emph{Dedicated to the memory of our good friend Ireneo Peral, whose enthusiasm for mathematics and life will always be a model for us.}
\end{center}

\section{Introduction and main results}
 \setcounter{equation}{0}

\subsection{Goal}
This paper is part of a project intending to give a precise description (decay/growth rates and profiles)  of the large-time behavior of solutions to the Cauchy  problem
\begin{equation}
\label{eq:fully.nl}
\partial_t^\a u+(-\Delta)^\beta u=f\quad\mbox{in }Q:=\R^N\times(0,\infty),\qquad
 	u(\cdot,0)=u_0\quad\mbox{in }\R^N,
\end{equation}
where $u_0$ and $f(\cdot,t)$ belong to $L^1(\mathbb{R}^N)$. Here $\partial_t^\alpha$, $\alpha\in(0,1)$, denotes the Caputo $\alpha$-derivative, introduced in~\cite{Caputo-1967}, defined for smooth functions by
$$
\displaystyle\partial_t^\alpha u(x,t)=\frac1{\Gamma(1-\alpha)}\,\partial_t\int_0^t\frac{u(x,\tau)- u(x,0)}{(t-\tau)^{\alpha}}\, {\rm d}\tau,
$$
and $(-\Delta)^\beta$, $\beta\in(0,1]$,  is the usual $\beta$ power of the Laplacian, defined for smooth functions by
$(-\Delta)^{s}=\mathcal{F}^{-1}(|\cdot|^{2s}\mathcal{F})$, where $\mathcal{F}$ stands for Fourier transform; see for instance~\cite{Stein-book-1970}.

\emph{Fully nonlocal} heat equations like~\eqref{eq:fully.nl},  nonlocal both in space and time, are useful to model situations with long-range interactions and memory effects, and have been proposed for example to describe plasma transport~\cite{delCastilloNegrete-Carreras-Lynch-2004,delCastilloNegrete-Carreras-Lynch-2005}; see also~\cite{Cartea-delCastilloNegrete-2007,Compte-Caceres-1998,Metzler-Klafter-2000,Zaslavsky-2002} for further models that use such equations.

When the forcing term $f$ is trivial, a complete description of the large-time behavior of~\eqref{eq:fully.nl} was recently given in~\cite{Cortazar-Quiros-Wolanski-2020a,Cortazar-Quiros-Wolanski-2020b}; see also~\cite{Kemppainen-Siljander-Zacher-2017}. Hence, since the problem is linear, it only remains to study the case with trivial initial datum, namely
\begin{equation}
\label{eq-f}
\partial_t^\a u+(-\Delta)^\beta u=f\quad\mbox{in }Q,\qquad
 	u(\cdot,0)=0\quad\mbox{in }\R^N.
\end{equation}
This task is by far more involved, and this paper represents a first step towards its completion. It is devoted to the obtention of (sharp) decay/growth rates  of solutions to~\eqref{eq-f} when the forcing term satisfies
 \begin{equation}\label{eq-hypothesis f}
 \|f(\cdot,t)\|_{L^1(\R^N)}\le \frac C{(1+t)^\gamma}\quad\text{for some }\gamma\in\mathbb{R}
 \end{equation}
 and  the spatial dimension is large, $N>4\beta$. This already involves critical phenomena depending on the values of $p$ and $\gamma$. If $1\le N\le 4\beta$, additional critical phenomena associated to the dimension appear, that make the analysis somewhat different. This case is considered in~\cite{Cortazar-Quiros-Wolanski-2021}. Notice that we are allowing $\gamma$ to take negative values, so that $\|f(\cdot,t)\|_{L^1(\mathbb{R}^N)}$ may grow with time.

If $f(\cdot,t)\in L^1(\mathbb{R}^N)$ for all $t\ge0$ and
$|\mathcal{F}f(\xi,t)|\le C|g(\xi)|$ for some function $g$ such that
$$
(1 + |\cdot|^\beta)g(\cdot) \in L^1(\mathbb{R}^N),
$$
then problem~\eqref{eq-f} has a unique bounded classical solution given by
Duhamel's type formula
\begin{equation}\label{eq-formula}
 		u(x,t)=\int_0^t\int_{\R^N}Y(x-y,t-s)f(y,s)\,{\rm d}y{\rm d}s,
\end{equation}
with $Y=\partial_t^{1-\a} Z$, where $Z$ is the solution to~\eqref{eq-f} with $f\equiv0$ having a Dirac mass as initial datum; see~\cite{Eidelman-Kochubei-2004,Kemppainen-Siljander-Zacher-2017}.
If we only assume $f(\cdot,t)\in L^1(\mathbb{R}^N)$, the function $u$ in~\eqref{eq-formula} is still well defined, but it is not in general a classical solution to~\eqref{eq-f}. Nevertheless, it is a solution in a generalized sense~\cite{Gripenberg-1985,Kemppainen-Siljander-Zacher-2017}.
In this paper we will always deal with solutions of this kind, given by~\eqref{eq-formula}, which are denoted in the literature as \emph{mild} solutions~\cite{Kemppainen-Siljander-Zacher-2017,Pruss-book}.

\medskip

\noindent\emph{Notation. } As is common in asymptotic analysis, $g\asymp h$ will mean that there are constants $\nu,\mu>0$ such that
$\nu h\le g\le \mu h$.

 \subsection{The kernel $Y$. Critical exponents}\label{subsect-estimates W}
Since the mild solution is given by the convolution in space and time of the forcing term $f$ with the kernel $Y$, having good estimates for the latter will be essential for the analysis. Such estimates were obtained in~\cite{Kemppainen-Siljander-Zacher-2017}, and are recalled next.

The kernel $Y$ has a self-similar form,
 \begin{equation}
 \label{eq:Y.selfsimilar}
 Y(x,t)=t^{\a-1-\frac{\a N}{2\beta}}G(\xi),\quad \xi=x t^{-\a/(2\beta)}.
 \end{equation}
Its profile $G$ is positive, radially symmetric and smooth outside the origin, and if $N>4\beta$ satisfies the sharp estimates
\begin{align}
\label{eq-interior estimates G}
 &G(\xi)\asymp 		{|\xi|^{4\beta-N}},&&|\xi|\le 1,\ \beta\in(0,1],\\
\label{eq-exterior estimates G1}
&G(\xi)\asymp |\xi|^{\frac{(N-2)(\alpha-1)}{(2-\alpha)}}\exp({-\sigma|\xi|^{\frac2{2-\a}}}),&&|\xi|\ge1,\ \beta=1, \\
\label{eq-exterior estimates Gbeta}
&G(\xi)\asymp |\xi|^{-(N+2\beta)},&&|\xi|\ge1,\ \beta\in(0,1).
\end{align}
In particular, we have the global bound
\begin{equation}
\label{eq:global.estimate.Y}
0\le Y(x,t)\le Ct^{-(1+\a)}|x|^{4\beta-N}\quad\text{in }Q,\quad \beta\in(0,1],
\end{equation}
and,   since $|\xi|^{\frac{(N-2)(\alpha-1)}{(2-\alpha)}}\exp(-\sigma|\xi|^{\frac2{2-\a}})\le C_\nu|\xi|^{-(N+2\beta)}$ if $|\xi|\ge\nu$,  also the exterior bound
\begin{equation}
  \label{eq:exterior.estimate.Y}
0\le Y(x,t)\le C_\nu t^{2\a-1}|x|^{-(N+2\beta)}\quad\text{if } |x|\ge \nu t^{\alpha/(2\beta)},\ t>0,\quad \beta\in(0,1].
\end{equation}
 Notice that $Y(\cdot,t)\in L^p(\mathbb{R}^N)$ if and only if $p\in [1,p_*)$, where $p_*:=N/(N-4\beta)$. Moreover,
 \begin{equation}\label{eq:p.norm.Y}
\|Y(\cdot,t)\|_{L^p(\mathbb{R}^N)}=C t^{\a-1-\frac{\a N}{2\beta}(1-\frac1p)}\quad\text{for all }t>0, \quad\text{if }p\in [1,p_*).
\end{equation}
Therefore, $Y\in L_{\rm loc}^1([0,\infty):L^p(\R^N))$ if and only if $p\in [1, p_{\rm c})$, where $p_{\rm c}:= N/(N-2\beta)$. Since the mild solution is given by a convolution of $f$ with $Y$ \emph{both in space and time}, the threshold value that will mark the border between subcritical and supercritical behaviors will be $p_{\rm c}$, and not $p_*$. In particular, condition~\eqref{eq-hypothesis f} guarantees that $u(\cdot,t)\in L^p (\mathbb{R}^N)$ for $p\in[1,p_{\rm c})$, but not for $p\ge p_{\rm c}$. Hence, in order to deal with supercritical exponents $p\ge p_{\rm c}$ we need some extra assumption on the spatial behavior of the forcing term. In the present paper we will use two different such extra hypotheses,  the poinwtise condition
\begin{equation}
\label{eq:decay.condition}
|f(x,t)|\le C|x|^{-N}(1+t)^{-\gamma}\quad \text{for }|x|\text{ large},
\end{equation}
and the integral condition
\begin{equation}
\label{eq:q.integrability.f}
\|f(\cdot,t)\|_{L^q(\R^N)}\le C(1+t)^{-\gamma}\quad\text{for some }q>q_{\rm c}(p):=\frac{N}{2\beta +\frac Np}.
\end{equation}
We do not claim that these conditions are optimal; but they are not too restrictive, and are easy enough to keep the proofs simple.

\subsection{Precedents and statement of results}\label{subsect-statement of results}
The only precedent is given in~\cite{Kemppainen-Siljander-Zacher-2017}, where  the authors study the problem in the \emph{integrable in time} case $\gamma>1$ and  prove, for all $p\in [1,\infty]$ if $1\le N<4\beta$, and for  $p\in [1,p_{\rm c})$ if $N\ge 4\beta$,   that
\begin{equation}
\label{eq:result.KSZ.subcritical}
\lim_{t\to\infty}t^{1-\alpha+\frac{\a N}{2\beta}(1-\frac1p)}\|u(\cdot,t)-M_\infty Y(\cdot,t)\|_{L^p(\mathbb{R}^N)}=0,\quad\text{where }M_\infty:=\int_0^\infty\int_{\mathbb{R}^N} f(x,t)\,{\rm d}x{\rm d}t<\infty.
\end{equation}
In particular, using~\eqref{eq:p.norm.Y} we get the sharp estimate
$$ 	
\|u(\cdot,t)\|_{L^p(\R^N)}\le C t^{-1+\a-\frac{\a N}{2\beta}(1-\frac1p)}.
$$
This result is also valid for the local case, $\alpha=1$; see for instance~\cite{Biler-Guedda-Karch-2004,Dolbeault-Karch-2006} for the case $p=1$. In this special situation $Y=Z$ is the well-known fundamental solution of the heat equation, whose profile does not have a spatial singularity and belongs to all $L^p$ spaces.

An analogous convergence result is definitely not possible for $\alpha\in(0,1)$ if $p\ge p_*$, since $Y(\cdot,t)\not\in L^p(\mathbb{R}^N)$ in that case, or if $\gamma\le 1$. Moreover, even in the subcritical range~\eqref{eq:result.KSZ.subcritical} only gives a sharp rate and a nontrivial limit profile in the diffusive scale $|x|\asymp t^{\alpha/(2\beta)}$; see below. Hence we need a different approach.

As we will see, it turns out that, in contrast with the local case, and due to the effect of memory, the decay/growth rates are not the same in different space-time scales. Moreover, the scale that determines the dominant rate depends on the value of the exponent $p$.
Our strategy will consist in tackling this difficulty directly by studying separately the rates in exterior regions, $|x|\ge\nu t^{\a/2\beta}$ with $\nu>0$, compact sets or intermediate  regions $ |x|\asymp g(t)$ with $g(t)\to\infty$ and $g(t)=o(t^{\a/2\beta})$.
We already found these phenomenon for the Cauchy problem, \eqref{eq:fully.nl} with $f\equiv0$, where the decay rate was  $O(t^{-\a})$ in compact sets and $O\big(t^{-\frac{N\a}{2\beta}(1-\frac1p)}\big)$ in exterior regions; see~\cite{Cortazar-Quiros-Wolanski-2020a,Cortazar-Quiros-Wolanski-2020b}.

Our first result concerns exterior regions.
\begin{teo}[\sc Exterior regions]
\label{teo-exterior} Let $f$ satisfy \eqref{eq-hypothesis f} and also \eqref{eq:decay.condition} if $p\ge p_{\rm c}$.  Let $u$ be the mild solution to \eqref{eq-f}. For all $\nu>0$ there is a constant $C$ such that
	\begin{equation}\label{eq-bound exterior}
	\|u(\cdot,t)\|_{L^p(\{|x|\ge \nu t^{\a/(2\beta)}\})}\le C \begin{cases}
	t^{-\gamma+\a-\frac{\a N}{2\beta}(1-\frac1p)},&\gamma<1,\\
t^{-1+\a-\frac{\a N}{2\beta}(1-\frac1p)}\log t,&\gamma=1,\\
t^{-1+\a-\frac{\a N}{2\beta}(1-\frac1p)},&\gamma>1.
\end{cases}\end{equation}
These estimates are sharp.
\end{teo}
For $p\in[1,p_{\rm c})$ and $\gamma>1$ the result follows from~\eqref{eq:result.KSZ.subcritical}, showing that the behavior in this regions dominates the global behavior in the subcritical case.

We now turn to the behavior in compact sets which, due to the effect of memory,  will dominate the global behavior for large values of $p$.

\begin{teo}[\sc Compact sets]\label{teo-compactos} Let  $f$ satisfy \eqref{eq-hypothesis f}. If $p\ge p_{\rm c}$, assume  also  \eqref{eq:q.integrability.f}  with $\gamma$ as in \eqref{eq-hypothesis f}. Let $u$ be the mild solution to \eqref{eq-f}. For every compact set $K$ there exists a constant $C$ such that
		\begin{equation}\label{eq-bound compactos1}
		\|u(\cdot,t)\|_{L^p(K)}\le C \begin{cases}
		t^{-\gamma},&\gamma\le 1+\a,\\
		t^{-(1+\a)},&\gamma\ge 1+\a.
		\end{cases}
		\end{equation}
These estimates are sharp.
\end{teo}

\noindent\emph{Remark.} Note that $q_{\rm c}(p_{\rm c}) =1$.

As expected, the rates in intermediate regions, between compact sets and exterior regions, are intermediate between the ones in such scales.
\begin{teo}[\sc  Intermediate regions]
\label{teo-intermediate}
Let $f$ satisfy \eqref{eq-hypothesis f} and also \eqref{eq:decay.condition} if $p\ge p_{\rm c}$. Let $g(t)\to\infty$ be such that $g(t)=o(t^{\a/(2\beta)})$. Let $u$ be the mild solution to \eqref{eq-f}. For all $0<\nu<\mu<\infty$ there exists a constant $C$ such that
	\begin{equation}\label{eq-bound intermediate}
	\|u(\cdot,t)\|_{L^p(\{\nu \le |x|/g(t)\le \mu\})}\le C g(t)^{2\beta-N(1-\frac1p)}\begin{cases}
	t^{-\gamma},&\gamma<1,\\
	\max\{t^{-1},t^{-(1+\a)}g(t)^{2\beta}\log t \},&\gamma=1,\\
	\max\{t^{-\gamma},t^{-(1+\a)} g(t)^{2\beta}\},&\gamma>1.
	\end{cases}
	\end{equation} 	
These estimates are sharp.
\end{teo}
%
%The above result yields the following corollary.
%\begin{coro}\label{coro-intermediate} Assume the hypotheses of Theorem~\ref{teo-intermediate}.  Let  $g(t)=t^\delta h(t)$ with $0\le\delta\le\a/(2\beta)$ and $\frac{|\log h(t)|}{\log t}\to0$ as $t\to \infty$. In case $\delta=0$, assume moreover that $h(t)\to\infty$ and if $\delta=\a/(2\beta)$, assume  moreover that $h(t)\to0$.  Then for every $0<\nu<\mu<\infty$,
%	\[\begin{aligned}
%\|u(\cdot,t)&\|_{L^p(\{\nu \le |x|/g(t)\le \mu\})}\\
%	&\le C\begin{cases}
%	t^{-\gamma+\delta\big(2\beta-N(1-\frac1p)\big)}h(t)^{2\beta-N(1-\frac1p)},&\gamma< 1+\a-2\beta\delta,\\
%		t^{-\gamma+\delta\big(2\beta-N(1-\frac1p)\big)}h(t)^{2\beta-N(1-\frac1p)},&\gamma= 1+\a-2\beta\delta\ \&\  \limsup\limits_{t\to\infty}h(t)<\infty,\\
%	t^{-(1+\a)+\delta\big(4\beta-N(1-\frac1p)\big)}h(t)^{4\beta-N(1-\frac1p)},&\gamma= 1+\a-2\beta\delta\ \&\  \liminf\limits_{t\to\infty}h(t)>0,\\
%t^{-(1+\a)+\delta\big(4\beta-N(1-\frac1p)\big)}h(t)^{4\beta-N(1-\frac1p)},&\gamma> 1+\a-2\beta\delta.
%	\end{cases}
%	\end{aligned}
%	\]
%\end{coro}

We also obtain results that connect the behaviors in  compact sets and exterior
regions, thus getting the (global) decay rate in $L^p(\R^N)$.
\begin{teo}[\sc Global results]
\label{teo-conexion}  Assume \eqref{eq-hypothesis f},  and also \eqref{eq:decay.condition}  and
\eqref{eq:q.integrability.f}
with $\gamma$ as in \eqref{eq-hypothesis f} if $p\ge p_{\rm c}$. Let $u$ be the mild solution to \eqref{eq-f}.  There is a constant $C$ such that $$	
\|u(\cdot,t)\|_{L^p(\R^N)}\le C\begin{cases}
		t^{-\gamma+\a-\frac{\a N}{2\beta}(1-\frac1p)},&\gamma<1,\\
		t^{-1+\alpha-\frac{\a N}{2\beta}(1-\frac1p)}\log t,&\gamma=1,\hskip3.4cm p\in [1,p_{\rm c}),\\
		t^{-1+\alpha-\frac{\a N}{2\beta}(1-\frac1p)},&\gamma>1,\\[0.3cm]
		t^{-\gamma}\log t,&\gamma\le 1,\\[-0.3cm]
&\hskip4.5cm p=p_{\rm c},\\[-0.3cm]
		t^{-1},&\gamma> 1,\\[0.3cm]
    	t^{-\gamma},&\gamma\le 1-\a+\frac{\a N}{2\beta}(1-\frac1p),\\[-0.3cm]
&\hskip4.5cm p\in (p_{\rm c},p_*),\\[-0.3cm]
    	t^{-1+\a-\frac{\a N}{2\beta}(1-\frac1p)},&\gamma\ge 1-\a+\frac{\a N}{2\beta}(1-\frac1p),\\[0.3cm]
		t^{-\gamma},&\gamma< 1+\a,\\[-0.3cm]
&\hskip4.5cm p=p_*,\\[-0.3cm]
		t^{-(1+\a)}\log t,&\gamma\ge 1+\a,\\[0.3cm]
		t^{-\gamma},&\gamma\le 1+\a,\\[-0.3cm]
&\hskip4.5cm p>p_*,\\[-0.3cm]
		t^{-(1+\a)},&\gamma\ge 1+\a;
		\end{cases}
$$
see Figure~\ref{fig:global.rates}.
These estimates are sharp.	
\end{teo}
		
\begin{figure}
  \includegraphics[width=\linewidth]{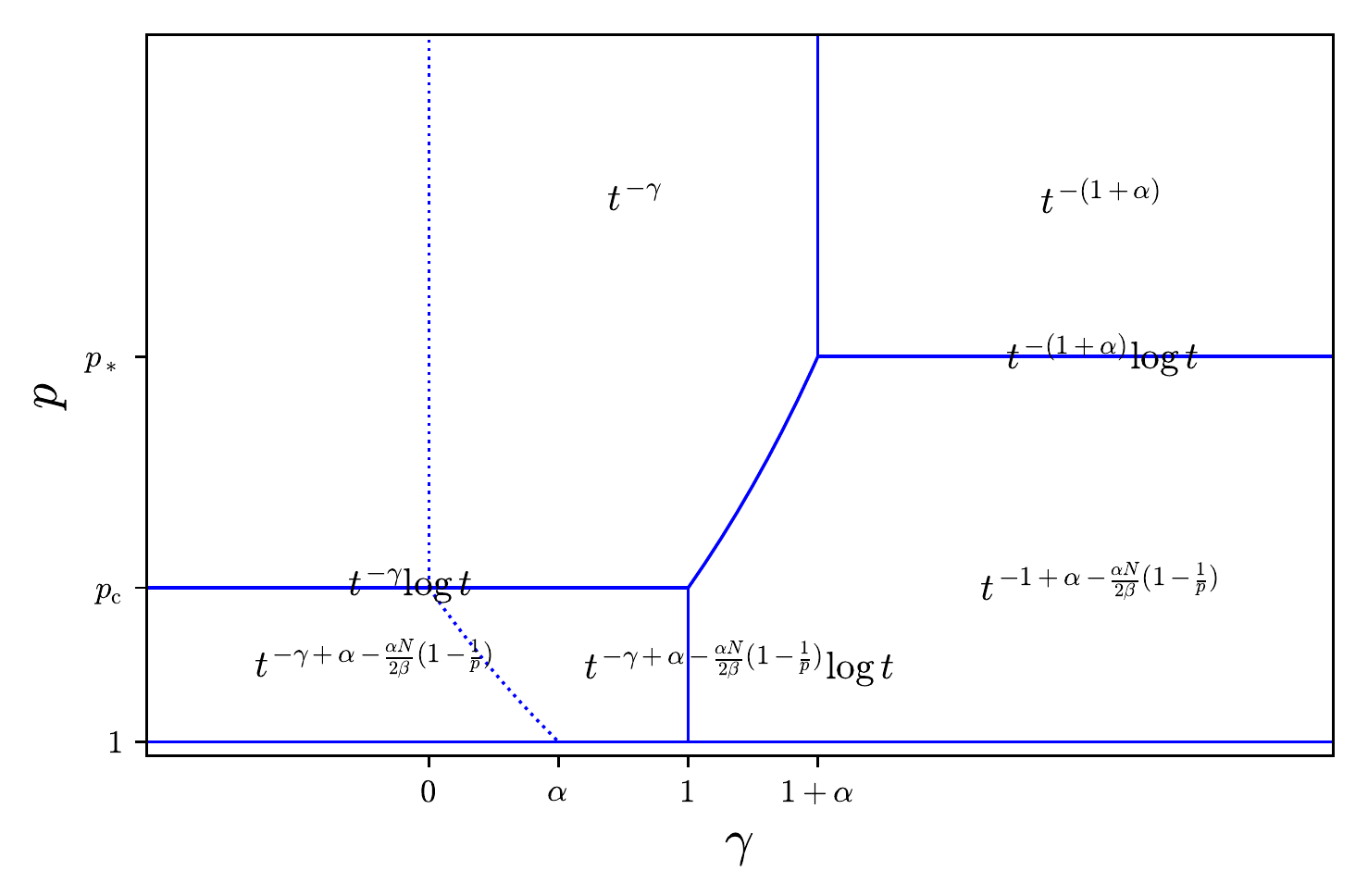}
  \vskip-0.2cm
  \caption{Global decay/growth rates. The dotted line indicates the borderline separating decay from growth.}
  \label{fig:global.rates}
\end{figure}

Notice that the borderline separating decay and growth is $\gamma=0$  only for $p\ge p_{\rm c}$. For $p\in [1,p_{\rm c})$ the frontier is given by
$$
\gamma=\a-\frac{\a N}{2\beta}(1-\frac1p);
$$
see the dotted line in~Figure~\ref{fig:global.rates}. For $p=1$ this corresponds to $\gamma=\alpha$. An informal explanation for this fact can be found in formula~\eqref{eq-formula}.  We are integrating in time, but $Y(x,t)=\partial^{1-\alpha}_t Z(x,t)$. Hence, it is like if we were integrating $\alpha$ times in time. As for the behavior for the borderline $\gamma$, in general there is neither growth nor decay. The exception is the case $p=p_{\rm c}$,  $\gamma=0$, in which there is a slow logarithmic growth.

Another remarkable fact is that the rates depend on $\gamma$ not only in the non-integrable case $\gamma\le1$, which might have been expected, but also in part of the region $\gamma\in [1,1+\alpha)$ if $p$ is supercritical.

\section{Exterior region}\label{sect-exterior}
 \setcounter{equation}{0}

In this section we prove Theorem~\ref{teo-exterior}, which gives the behavior of all $L^p$ norms of the mild solution $u$ to \eqref{eq-f} in exterior regions, $\{(x,t)\in Q:|x|\ge \nu t^{\a/(2\beta)}\}$, $\nu>0$.

\begin{proof}[Proof of Theorem~\ref{teo-exterior}.]
The starting point is Duhamel's type formula~\eqref{eq-formula}. If  $p\in[1,p_{\rm c})$,  then $Y(\cdot, t)$ belongs to $L^p(\R^N)$. Therefore, using  \eqref{eq-hypothesis f} and~\eqref{eq:p.norm.Y},
	\[\begin{aligned}
	\|u(\cdot,t)\|_{L^p(\R^N)}&\le \int_0^t\|Y(\cdot,t-s)\|_{L^p(\R^N)}\|f(\cdot,s)\|_{L^1(\R^N)}\,{\rm d}s\\
	&\le C\int_0^t(t-s)^{\a-1-\frac{N\a}{2\beta}(1-\frac1p)}(1+s)^{-\gamma}\,{\rm d}s\\
	&\le C t^{\a-1-\frac{N\a}{2\beta}(1-\frac1p)}\int_0^{\frac t2}(1+s)^{-\gamma}\,{\rm d}s+ C t^{-\gamma}\int_{\frac t2}^t (t-s)^{\a-1-\frac{N\a}{2\beta}(1-\frac1p)}\,{\rm d}s\\
	&\le  C t^{\a-1-\frac{N\a}{2\beta}(1-\frac1p)}\int_0^{\frac t2}(1+s)^{-\gamma}\,{\rm d}s+ Ct^{\a-\gamma-\frac{N\a}{2\beta}(1-\frac1p)}.
	\end{aligned}
	\]
Integration of $\displaystyle \int_0^{\frac t2}(1+s)^{-\gamma}\,{\rm d}s$ gives the result.
	
We turn now our attention to the case  $p\ge p_{\rm c}$, for which we assume also the decay condition~\eqref{eq:decay.condition}.  We have $|u|\le {\rm I}+{\rm II}$, where
	\[\begin{aligned}
	{\rm I}(x,t)&=\int_0^t\int_{\{|y|<\frac{|x|}2\}}Y(x-y,t-s)|f(y,s)|\,{\rm d}y{\rm d}s,\\
	{\rm II}(x,t)&=\int_0^t\int_{\{|y|>\frac{|x|}2\}}Y(x-y,t-s)|f(y,s)|\,{\rm d}y{\rm d}s,\\
	\end{aligned}\]
	
We start by estimating ${\rm I}$.  Notice that if $|y|<|x|/2$, then  $|x-y|>|x|/2$. Thus, if moreover $|x|\ge\nu t^{\a/(2\beta)}$, there holds that $|x-y|(t-s)^{-\a/(2\beta)}>\nu /2$. Hence, the   bound \eqref{eq:exterior.estimate.Y} yields
	\begin{equation*}
	Y(x-y,t-s)\le C(t-s)^{2\a-1}|x-y|^{-(N+2\beta)}\le C(t-s)^{2\a-1}|x|^{-(N+2\beta)}.
	\end{equation*}
Using also~\eqref{eq-hypothesis f} we arrive at $\displaystyle{\rm I}(x,t)\le C |x|^{-(N+2\beta)}\int_0^t(t-s)^{2\a-1}(1+s)^{-\gamma}\,{\rm d}s$, and
therefore
\[
\begin{aligned}
&\|{\rm I}(\cdot,t)\|_{L^p(\{|x|\ge\nu t^{\a/(2\beta)}\})} \\
&\qquad\le C t^{-\a-\frac{\a N}{2\beta}(1-\frac1p)}\left(\int_0^{\frac t2}(1+s)^{-\gamma}(t-s)^{2\a-1}\,{\rm d}s
+\int_{\frac t2}^t(1+s)^{-\gamma}(t-s)^{2\a-1}\,{\rm d}s\right)\\
&\qquad\le C t^{\a-1-\frac{\a N}{2\beta}(1-\frac1p)}\int_0^{\frac t2}(1+s)^{-\gamma}\,{\rm d}s+C t^{-\gamma-\a-\frac{\a N}{2\beta}(1-\frac1p)}\int_{\frac t2}^t(t-s)^{2\a-1}\,{\rm d}s\\
&\qquad\le C t^{\a-1-\frac{\a N}{2\beta}(1-\frac1p)}\int_0^{\frac t2}(1+s)^{-\gamma}\,{\rm d}s+C t^{\a-\gamma-\frac{\a N}{2\beta}(1-\frac1p)},
\end{aligned}	
\]
and the desired bound for ${\rm I}$ follows.
	
Now we turn to ${\rm II}$. We choose $\ep\in(0,1/2)$. Given $x\in\mathbb{R}^N$, $t>0$,  $s\in(0,t)$, we denote $\mathcal{B}(x,t,s)=\{y\in\mathbb{R}^N: |x-y|<|x|^\ep(t-s)^{(\alpha/(2\beta))(1-\ep)}\}$. Note that for $|x|>\nu t^{\a/(2\beta)}$ and $y\in(\mathcal{B}(x,t,s))^{\rm c}$ we have $|x-y|(t-s)^{-\alpha/(2\beta)}>\nu^\ep$. Therefore, using the estimates~\eqref{eq:global.estimate.Y}--\eqref{eq:exterior.estimate.Y}, we have ${\rm II}\le {\rm II}_1+{\rm II}_2$, with
		\[\begin{aligned}
		{\rm II}_1(x,t)&= C\int_0^t\int_{\{|y|>\frac{|x|}{2}\}\cap \mathcal{B}(x,t,s)}(t-s)^{-(1+\a)}|x-y|^{4\beta-N}|f(y,s)|\,{\rm d}y{\rm d}s,\\
		{\rm II}_2(x,t)&=C\int_0^t\int_{\{|y|>\frac{|x|}{2}\}\cap (\mathcal{B}(x,t,s))^{\rm c}}(t-s)^{2\a-1}|x-y|^{-(N+2\beta)}|f(y,s)|\,{\rm d}y{\rm d}s.
		\end{aligned}\]
				We have, using the decay condition~\eqref{eq:decay.condition},
		\[\begin{aligned}
		{\rm II}_1(x,t)&\le C\int_0^t\int_{\{|y|>\frac{|x|}{2}\}\cap \mathcal{B}(x,t,s)}|y|^{-N}(1+s)^{-\gamma}(t-s)^{-(1+\a)}|x-y|^{4\beta-N}\,{\rm d}y{\rm d}s\\
		&\le C |x|^{-N}\int_0^t(1+s)^{-\gamma}(t-s)^{-(1+\a)}\int_{\mathcal{B}(x,t,s)}|x-y|^{4\beta-N}\,{\rm d}y{\rm d}s\\
		&=C|x|^{4\beta\ep-N}\int_0^t(1+s)^{-\gamma}(t-s)^{\a(1-2\ep)-1}\,{\rm d}s.
		\end{aligned}
\]
Since $p\ge p_{\rm c}$ and $\ep\in(0,1/2)$, we have $(N-4\beta\ep)p>N$, and hence
\[
\begin{aligned}
\| {\rm II}_1(\cdot,t)\|_{L^p(\{|x|>\nu t^{\a/(2\beta)}\})}
&= C t^{2\ep\a-\frac{\a N}{2\beta}(1-\frac1p)}\int_0^t(1+s)^{-\gamma}(t-s)^{\a(1-2\ep)-1}\,{\rm ds}\\
&\le  C t^{\a-1-\frac{\a N}{2\beta}(1-\frac1p)}\int_0^{\frac t2}(1+s)^{-\gamma}\,{\rm d}s+Ct^{\a-\gamma-\frac{\a N}{2\beta}(1-\frac1p)},
\end{aligned}
\]
and integration gives the  bound in \eqref{eq-bound exterior} for this term.

Finally, since $|x-y|\ge|x|^\ep(t-s)^{(\alpha/(2\beta))(1-\ep)}$ in $(\mathcal{B}(x,t,s))^{\rm c}$, using the condition~\eqref{eq-hypothesis f} on $f$,
\[{\rm II}_2(x,t)\le C|x|^{-(N+2\beta)\ep}\int_0^t(1+s)^{-\gamma}(t-s)^{\a-1-\frac{\a N}{2\beta}+\frac{\a }{2\beta}(N+2\beta)\ep}\,{\rm d}s,\]
so that
\[
\begin{aligned}
\|{\rm II}_2(\cdot,t)\|_{L^p(\{|x|>\nu t^{\a/(2\beta)}\})}
&\le C t^{\frac{\a N}{2\beta p}-\frac{\a}{2\beta}(N+2\beta)\ep}
\int_0^t(1+s)^{-\gamma}(t-s)^{\a-1-\frac{\a N}{2\beta}+\frac{\a }{2\beta}(N+2\beta)\ep}\,{\rm d}s\\
&\le C t^{\a-1-\frac{\a N}{2\beta}(1-\frac1p)}\int_0^{\frac t2}(1+s)^{-\gamma}\,{\rm d}s+ C t^{\a-\gamma-\frac{a N}{2\beta}(1-\frac1p)}.
\end{aligned}
\]
Now, integration gives the  bound.

In order to check that the bound in \eqref{eq-bound exterior} is sharp we choose $f(x,t)=(1+t)^{-\gamma}\chi_{B_1}(x)$. If $t$ is large, $|y|<1$ and $|x|>\nu t^{\a/(2\beta)}$, then $|y|<|x|/2$. Hence, $|x|/2<|x-y|<3|x|/2$, so that, assuming also that $|x|<\mu t^{\a/(2\beta)}$ and $0<s<t/2$,
		\[
		\frac\nu2\le \frac{|x|}{2t^{\frac\a{2\beta}}}\le\frac{|x-y|}{(t-s)^{\frac\a{2\beta}}}\le\frac{3\mu}2\Big(\frac t{t-s}\Big)^{\frac\a{2\beta}}
\le C.
\]
Thus, since the profile $G$ of $Y$ is positive, under these conditions $Y(x-y,t-s)\ge C(t-s)^{\alpha-1-\frac{N\a}{2\beta}}$ for some constant $C>0$, see~\eqref{eq:Y.selfsimilar}, and therefore,
		\[
		u(x,t)\ge C\int_0^{\frac t2}(1+s)^{-\gamma}(t-s)^{\a-1-\frac{\a N}{2\beta}}\,{\rm d}s\ge Ct^{\a-1-\frac{\a N}{2\beta}}\int_0^{\frac t2}(1+s)^{-\gamma}\,{\rm d}s. \]
Thus,
\[
\begin{aligned}
		\|u(\cdot,t)\|_{L^p(\{|x|>\nu t^{\a/(2\beta)}\})}&\ge  \|u(\cdot,t)\|_{L^p(\{\mu >|x|/t^{\a/(2\beta)}>\nu \})}\\
&\ge Ct^{\a-1-\frac{\a N}{2\beta}}|\{\nu t^{\a/(2\beta)}<|x|<\mu t^{\a/(2\beta)}\}\big|^{1/p}\int_0^{\frac t2}(1+s)^{-\gamma}\,{\rm d}s\\
		&=Ct^{\a-1-\frac{\a N}{2\beta}(1-\frac1p)}\int_0^{\frac t2}(1+s)^{-\gamma}\,{\rm d}s,
\end{aligned}
\]
which implies the desired lower bound.		
\end{proof}

\section{Compact regions}\label{sect-compacts}
\setcounter{equation}{0}

In this section we prove Theorem~\ref{teo-compactos}, which gives the large-time behavior of the $L^p$ norms of the mild solution to \eqref{eq-hypothesis f} in compact sets $K$.

\begin{proof}[Proof of Theorem \ref{teo-compactos}]
Let $t\ge 1$.
We have $|u|\le {\rm I}+{\rm II}$, where
	\[\begin{aligned}
	{\rm I}(x,t)&=\int_0^{t-1}\int_{\R^N} Y(x-y,t-s) |f(y,s)|\,{\rm d}y{\rm d}s,\\
	{\rm II}(x,t)&=\int_{t-1}^t\int_{\R^N}  Y(x-y,t-s)|f(y,s)|\,{\rm d}y{\rm d}s.\\
	\end{aligned}\]	
Using the global bound~\eqref{eq:global.estimate.Y} for $Y$ we get
		\[\begin{aligned}
		{\rm I}(x,t)\le&\, C\int_0^{t-1}(t-s)^{-(1+\a)}\int_{\R^N}|x-y|^{4\beta- N}|f(y,s)|\,{\rm d}y{\rm d}s\\
		\le&\, C\int_0^{t-1}(t-s)^{-(1+\a)}\int_{\{|x-y|<1\}}|x-y|^{4\beta-N}|f(y,s)|\,{\rm d}y{\rm d}s\\	
&+		C\int_0^{t-1}(t-s)^{-(1+\a)}\int_{\{|x-y|>1\}}{|f(y,s)|}\,{\rm d}y{\rm d}s.
		\end{aligned}
		\]
Let $q=1$ if $p\in[1,p_{\rm c})$, $q>q_{\rm c}(p)$ as in~\eqref{eq:q.integrability.f} if $p\ge p_{\rm c}$.  Let $r$ satisfy $1+\frac1p=\frac1q+\frac1r$. Then $r\in [1,p_{\rm c})$, and in particular $r\in [1, p_*)$. Thus, for all $t\ge 2$ we have
$$
\begin{aligned}
		\|{\rm I}(\cdot,t)\|_{L^p(K)}\le&\, C \int_0^{t-1}(t-s)^{-(1+\a)}\|f(\cdot,s)\|_{L^q(K+B_1)}\Big(\int_{B_1} |z|^{(4\beta-N)r}\,{\rm d}z\Big)^{1/r}\,{\rm d}s\\
		&+C|K|^{1/p}\int_0^{t-1}(t-s)^{-(1+\a)}\|f(\cdot,s)\|_{L^1(\R^N)}\,{\rm d}s\\
		\le&\, C\int_0^{t-1}(1+s)^{-\gamma}(t-s)^{-(1+\a)}\,{\rm d}s\\
		\le&\, Ct^{-(1+\a)}\int_0^{\frac t2}(1+s)^{-\gamma}\,{\rm d}s+Ct^{-\gamma}\int_{\frac t2}^{t-1}(t-s)^{-(1+\a)}\,{\rm d}s\\
		\le&\, Ct^{-\gamma}+C\begin{cases}
		t^{-(\gamma+\a)},&\gamma<1,\\
		t^{-(1+\a)}\log t,&\gamma=1,\\
		t^{-(1+\a)},&\gamma>1,
		\end{cases}
\end{aligned}
$$
which implies that
\begin{equation*}
\label{eq-bound compactos I}
		\|{\rm I}(\cdot,t)\|_{L^p(K)}
\le	C\begin{cases}
		t^{-\gamma},&\gamma<1+\a,\\
		t^{-(1+\a)},&\gamma\ge1+\a.
		\end{cases}
\end{equation*}

In order to bound ${\rm II}$ we take $r\in[1,p_{\rm c})$ as before. Then, using~\eqref{eq:p.norm.Y} we get
\begin{equation*}
\label{eq-bound compactos II}
    \begin{aligned}
    \|{\rm II}(\cdot,t)\|_{L^p(K)}
    &\le C\int_{t-1}^t\|f(\cdot,s)\|_{L^q(\R^N)}(t-s)^{\a-1-\frac{\a N}{2\beta}(1-\frac1r)}\,{\rm d}s\\
    &\le C\int_{t-1}^t(1+s)^{-\gamma}(t-s)^{\a-1-\frac{\a N}{2\beta}(1-\frac1r)}\,{\rm d}s\\
    &\le C t^{-\gamma}\int_0^1\tau^{\a-1-\frac{\a N}{2\beta}(1-\frac1r)}\,{\rm d}\tau=C t^{-\gamma},
    \end{aligned}	
\end{equation*}
which combined with the estimate for ${\rm I}$ yields the result.

In order to prove that estimate~\eqref{eq-bound compactos1}  is sharp we consider $f(x,t)=(1+t)^{-\gamma}\chi_{K+B_1}(x)$, where~$K$ is any compact set with measure different from 0.  We have
\[u(x,t)\ge \int_{0}^{t-1}(1+s)^{-\gamma}\int_{K+B_1}Y(x-y,t-s)\,{\rm d}y{\rm d}s.\]
If $x\in K$ and $|x-y|<1$, then $y\in K+B_1$. Notice that $|x-y|<1$ and $s<t-1$ imply that
$|x-y|(t-s)^{-\a/(2\beta)}\le 1$. Therefore, using the self-similar form~\eqref{eq:Y.selfsimilar} of $Y$ and the bound from below~\eqref{eq-interior estimates G} for the profile $G$,  for all $x\in K$ we have
\[
\begin{aligned}
u(x,t)&\ge C\int_{0}^{t-1}(1+s)^{-\gamma}(t-s)^{-(1+\a)}\int_{\{|x-y|<1\}}|x-y|^{4\beta-N}\,{\rm d}y{\rm d}s	\\
&=C\int_{0}^{t-1}(1+s)^{-\gamma}(t-s)^{-(1+\a)}\,{\rm d}s
\end{aligned}
\]
for some constant $C>0$. Thus, no matter the value of $\gamma$, for all $x\in K$ and $t$ large enough,
$$
u(x,t)\ge C
t^{-\gamma}\int_{\frac t2}^{t-1}(t-s)^{-(1+\a)}\,{\rm d}s=Ct^{-\gamma}\int_{1}^{\frac t2}\tau^{-(1+\a)}\,{\rm d}\tau
\ge C t^{-\gamma},
$$
while if $\gamma>1$, then
$$
u(x,t)\ge C
\int_0^{\frac t2}(1+s)^{-\gamma}(t-s)^{-(1+\a)}\,{\rm d}s
\ge C t^{-(1+\a)}\int_0^{\frac t2}(1+s)^{-\gamma}\,{\rm d}s\ge C t^{-(1+\a)},
$$
so that
\[
\|u(\cdot,t)\|_{L^p(K)}\ge C\begin{cases}
t^{-\gamma},&\gamma\le 1+\a,\\
t^{-(1+\a)},&\gamma\ge 1+\a.
\end{cases}
\]
\end{proof}

 \section{Intermediate scales}\label{sect-intermediate}
\setcounter{equation}{0}
In this section we study the large-time behavior of the $L^p$ norms of the mild solution to \eqref{eq-hypothesis f} in regions where $|x|\asymp g(t)$ with $g(t)\to\infty$ such that $g(t)=o(t^{\a/(2\beta)})$, which is the content of Theorem~\ref{teo-intermediate}.

\begin{proof}[Proof of Theorem~\ref{teo-intermediate}.]
We have $|u|\le {\rm I}+{\rm II}$, where
	\begin{equation}\label{eq-decomposition intermediate1}\begin{aligned}
	{\rm I}(x,t)&=\int_0^t\int_{\{|y|>\frac{|x|}2\}}|f(x-y,t-s)|Y(y,s)\,{\rm d}y{\rm d}s,\\
	{\rm II}(x,t)&=\int_0^t\int_{\{|y|<\frac{|x|}2\}}|f(x-y,t-s)|Y(y,s)\,{\rm d}y{\rm d}s.
	\end{aligned}
	\end{equation}
	
To estimate ${\rm I}$ we decompose it as ${\rm I}={\rm I}_1+{\rm I}_2+{\rm I}_3$, where
	\begin{equation}\label{eq-decomposition intermediate2}\begin{aligned}
	{\rm I}_1(x,t)&= \int_0^{\big(\frac{|x|}2\big)^{2\beta/\a}}\int_{\{|y|>\frac{|x|}2\}}|f(x-y,t-s)|Y(y,s)\,{\rm d}y{\rm d}s,\\
	{\rm I}_2(x,t)&=\int_{\big(\frac{|x|}2\big)^{2\beta/\a}}^{\frac t2}\int_{\{|y|>\frac{|x|}2\}}|f(x-y,t-s)|Y(y,s)\,{\rm d}y{\rm d}s,\\
	{\rm I}_3(x,t)&= \int_{\frac t2}^{t}\int_{\{|y|>\frac{|x|}2\}}|f(x-y,t-s)|Y(y,s)\,{\rm d}y{\rm d}s.
	\end{aligned}
	\end{equation}
If  $0<s^{\a/(2\beta)}<|x|/2<|y|$, then $|y|s^{-\a/(2\beta)}\ge1$. Therefore, using~\eqref{eq:exterior.estimate.Y} and condition~\eqref{eq-hypothesis f}, if $|x|<\mu g(t)$ with $g(t)=o(t^{\a/(2\beta)})$ we have
\begin{equation}
\label{eq:intermediate.I1}
	{\rm I}_1(x,t)\le C  |x|^{-(N+2\beta)}\int_0^{\big(\frac{|x|}2\big)^{2\beta/\a}} s^{2\a-1}(1+t-s)^{-\gamma}\,{\rm d}s\le C |x|^{2\beta-N}t^{-\gamma}.
\end{equation}
Thus,
	\begin{equation*}\label{eq-bound intermediate I1beta}
	\|{\rm I}_1(\cdot,t)\|_{L^p(\{\nu <|x|/g(t)<\mu\})}\le C t^{-\gamma}g(t)^{2\beta-N(1-\frac1p)}.
	\end{equation*}

As for ${\rm I}_2$ and ${\rm I}_3$,  using the global bound~\eqref{eq:global.estimate.Y} for $Y$  and condition~\eqref{eq-hypothesis f},
\begin{equation}
\label{eq:intermediate.I2.I3}
\begin{aligned}
	{\rm I}_2(x,t)&\le C |x|^{4\beta-N}\int_{\big(\frac{|x|}2\big)^{2\beta/\a}}^{\frac t2}(1+t-s)^{-\gamma}s^{-(1+\a)}\,{\rm d}s\\
	&\le C|x|^{4\beta-N}t^{-\gamma}\int_{\big(\frac{|x|}2\big)^{2\beta/\a}}^{\frac t2}s^{-(1+\a)}\,ds\le C|x|^{2\beta-N}t^{-\gamma},\\
{\rm I}_3(x,t)&\le C|x|^{4\beta-N}t^{-(1+\a)}\int_{0}^{\frac t2}(1+r)^{-\gamma}\,{\rm d}r.
	\end{aligned}
\end{equation}
Therefore,
	\begin{equation*}\label{eq-bound intermediate I2}
    \begin{aligned}
	\|{\rm I}_2(\cdot,t)\|_{L^p(\{\nu <|x|/g(t)<\mu\})}&\le C t^{-\gamma} g(t)^{2\beta-N(1-\frac1p)},\\
    \|{\rm I}_3(\cdot,t)\|_{L^p(\{\nu <|x|/g(t)<\mu\})}&\le C g(t)^{4\beta-N(1-\frac1p)}\begin{cases}
	t^{-(\gamma+\a)},&\gamma<1,
\\
	t^{-(1+\a)}\log t, &\gamma=1,\\
	t^{-(1+\a)},&\gamma>1.
	\end{cases}
	\end{aligned}
	\end{equation*}

Now we turn to ${\rm II}$. We decompose it as ${\rm II}\le {\rm II}_1+{\rm II}_2$, where
\begin{equation}
\label{eq-decomposition intermediate3}\begin{aligned}
	{\rm II}_{1}(x,t)&=C\int_0^{\frac t2}\int_{\{|y|<\frac{|x|}2\}}|f(x-y,t-s)|Y(y,s)\,{\rm d}y{\rm d}s,\\	
	{\rm II}_2(x,t)&=C\int_{\frac t2}^t\int_{\{|y|<\frac{|x|}2\}}|f(x-y,t-s)|Y(y,s)\,{\rm d}y{\rm d}s.
	\end{aligned}
\end{equation}

We start with the subcritical case  $p\in[1,p_{\rm c})$.  Notice that if $s^{\a/(2\beta)}<|y|$, then $|y|s^{-\a/(2\beta)}\ge 1$. Moreover,  if  $|x|=o(t^{\a/(2\beta)})$, then $(|x|/2)^{2\beta/\a}=o(t)$, and hence $(|x|/2)^{2\beta/\a}<t/2$ if $t$ is large. Therefore, using the bounds~\eqref{eq:global.estimate.Y} and~\eqref{eq:exterior.estimate.Y}, we have ${\rm II}_1\le {\rm II}_{11}+{\rm II}_{12}$, where
\begin{equation*}
\label{eq:decomposition.intermediate.II}\begin{aligned}
	{\rm II}_{11}(x,t)&=C\int_0^{\big(\frac{|x|}2\big)^{2\beta/\a}}\int_{\{s^{\a/(2\beta)}<|y|<\frac{|x|}2\}}|f(x-y,t-s)|s^{2\a-1}|y|^{-(N+2\beta)}\,{\rm d}y{\rm d}s,\\	
{\rm II}_{12}(x,t)&=C\int_0^{\frac t2}\int_{\{|y|<\min\{\frac{|x|}2,s^{\a/(2\beta)}\}\}}|f(x-y,t-s)|s^{-(1+\a)}|y|^{4\beta-N}\,{\rm d}y{\rm d}s.
%\\
%&=C\int_0^{\big(\frac{|x|}2\big)^{2\beta/\a}}\int_{\{|y|<s^{\a/(2\beta)}\}}|f(x-y,t-s)|s^{-(1+\a)}|y|^{4\beta-N}\,{\rm d}y{\rm d}s\\
%	&+C\int_{\big(\frac{|x|}2\big)^{2\beta/\a}}^{\frac t2}\int_{\{|y|<\frac{|x|}2\}}|f(x-y,t-s)|s^{-(1+\a)}|y|^{4\beta-N}\,{\rm d}y{\rm d}s.
	\end{aligned}
\end{equation*}
Using condition~\eqref{eq-hypothesis f}, and remembering that $g(t)=o(t^{\a/(2\beta)})$,
	\[\begin{aligned}
	\|{\rm II}_{11}&(\cdot,t)\|_{L^p(\{\nu <|x|/g(t)<\mu\})}\\
\le&\,
	C\int_0^{\big(\frac{|x|}2\big)^{2\beta/\a}}\|f(\cdot,t-s)\|_{L^1(\R^N)}s^{2\a-1}\Big(\int_{\{|y|>s^{\a/(2\beta)}\}} |y|^{-(N+2\beta)p}\,{\rm d}y\Big)^{1/p}\,{\rm d}s\\
	\le&\, C\int_0^{\big(\frac{|x|}2\big)^{2\beta/\a}}(1+t-s)^{-\gamma}s^{\a-1-\frac{\a N}{2\beta}(1-\frac1p)}\,{\rm d}s\le C t^{-\gamma}|x|^{2\beta-N(1-\frac1p)}\le  C t^{-\gamma}g(t)^{2\beta-N(1-\frac1p)},\\
\end{aligned}
\]
\[\begin{aligned}
\|{\rm II}_{12}&(\cdot,t)\|_{L^p(\{\nu <|x|/g(t)<\mu\})}\\
\le&\, C\int_0^{\big(\frac{|x|}2\big)^{2\beta/\a}}\|f(\cdot,t-s)\|_{L^1(\R^N)}s^{-(1+\a)}
\Big(\int_{\{|y|<s^{\a/(2\beta)}\}}|y|^{(4\beta-N)p}\,{\rm d}y\Big)^{1/p}\,{\rm d}s
\\&+ C\int_{\big(\frac{|x|}2\big)^{2\beta/\a}}^{\frac t2}\|f(\cdot,t-s)\|_{L^1(\R^N)}s^{-(1+\a)}\Big(\int_{\{|y|<\frac{|x|}2\}} |y|^{(4\beta-N)p}\,{\rm d}y\Big)^{1/p}\,{\rm d}s\\
\le&\, C t^{-\gamma}\int_0^{\big(\frac{|x|}2\big)^{2\beta/\a}}s^{\a-1-\frac{\a N}{2\beta}(1-\frac1p)}\,{\rm d}s+ C|x|^{4\beta-N(1-\frac1p)}t^{-\gamma}\int_{\big(\frac{|x|}2\big)^{2\beta/\a}}^{\frac t2}s^{-(1+\a)}\,{\rm d}s\\
\le&\, C t^{-\gamma}|x|^{2\beta-N(1-\frac1p)}\le C t^{-\gamma}g(t)^{2\beta-N(1-\frac1p)}.
\end{aligned}
\]

On the other hand, from the global bound~\eqref{eq:global.estimate.Y},
$$
	{\rm II}_2(x,t)=C\int_{\frac t2}^t\int_{\{|y|<\frac{|x|}2\}}|f(x-y,t-s)|s^{-(1+\a)}{|y|^{4\beta-N}}\,{\rm d}y{\rm d}s,
$$
and therefore, thanks to condition~\eqref{eq-hypothesis f},
	\begin{equation*}\label{eq-bound intermediate II3}\begin{aligned}
	\|{\rm II}_2&(\cdot,t)\|_{L^p(\{\nu <|x|/g(t)<\mu\})}\\
&\le \int_{\frac t2}^t\|f(\cdot,t-s)\|_{L^1(\R^N)}s^{-(1+\a)}\Big(\int_{\{|y|<\frac{|x|}2\}} |y|^{(4\beta-N)p}\,{\rm d}y\Big)^{1/p}\,{\rm d}s\\
    &\le  C |x|^{4\beta-N(1-\frac1p)}t^{-(1+\a)}\int_0^{\frac t2}(1+r)^{-\gamma}\,{\rm d}r\le Cg(t)^{4\beta-N(1-\frac1p)}\begin{cases}
	t^{-(\gamma+\a)},&\gamma<1,\\
	t^{-(1+\a)}\log t,&\gamma=1,\\
	t^{-(1+\a)},&\gamma>1.
	\end{cases}
	\end{aligned}\end{equation*}

Let now $p\ge p_{\rm c}$. Since $|x-y|\ge|x|/2\ge \nu g(t)/2\to \infty$ as $t\to\infty$, then, thanks to assumption~\eqref{eq:decay.condition}, we have that $|f(x-y,t-s)|\le C|x|^{-N}(1+t-s)^{-\gamma}$ for all $t$ large. Hence,
	\[\begin{aligned}
	{\rm II}_{11}(x,t)
	\le&\, C|x|^{-N}\int_0^{\big(\frac{|x|}2\big)^{2\beta/\a}}(1+t-s)^{-\gamma}s^{2\a-1}\int_{\{|y|>s^{\a/(2\beta)}\}}|y|^{-(N+2\beta)}\,{\rm d}y{\rm d}s\\
	=&\,C|x|^{-N}\int_0^{\big(\frac{|x|}2\big)^{2\beta/\a}}(1+t-s)^{-\gamma}s^{\a-1}\,{\rm d}s\\
	\le&\, C |x|^{-N}t^{-\gamma}\int_0^{\big(\frac{|x|}2\big)^{2\beta/\a}}s^{\a-1}\,{\rm d}s=C|x|^{2\beta-N}t^{-\gamma},\\
	{\rm II}_{12}(x,t)\le&\, C |x|^{-N}\Big(\int_0^{\big(\frac{|x|}2\big)^{2\beta/\a}}(1+t-s)^{-\gamma}s^{-(1+\a)}\int_{\{|y|<s^{\a/(2\beta)}\}}|y|^{4\beta-N}\,{\rm d}y{\rm d}s\\
&+\int_{\big(\frac{|x|}2\big)^{2\beta/\a}}^{\frac t2}
	(1+t-s)^{-\gamma}s^{-(1+\a)}\int_{\{|y|<\frac{|x|}2\}}|y|^{4\beta-N}\,{\rm d}y{\rm d}s\Big)\\
	\le&\, C|x|^{-N} t^{-\gamma}\int_0^{\big(\frac{|x|}2\big)^{2\beta/\a}}s^{\a-1}\,{\rm d}s+
	C|x|^{4\beta-N}t^{-\gamma}\int_{\big(\frac{|x|}2\big)^{2\beta/\a}}^{\frac t2}s^{-(1+\a)}\,{\rm d}s\le C|x|^{2\beta-N}t^{-\gamma}.
	\end{aligned}\]	
Therefore,	$\|{\rm II}_{1}(\cdot,t)\|_{L^p(\{\nu <|x|/g(t)<\mu\})}\le 	C t^{-\gamma}g(t)^{2\beta-N(1-\frac1p)}$
also when $p\ge p_{\rm c}$.

As for ${\rm II}_2$, also when $p\ge p_{\rm c}$, since $|y|<|x|/2$ implies $|x-y|>|x|/2$, using the global estimate~\eqref{eq:global.estimate.Y} and the decay condition~\eqref{eq:decay.condition},
		\[\begin{aligned}
	{\rm II}_2(x,t)&\le C |x|^{-N}\int_{\frac t2}^t(1+t-s)^{-\gamma} s^{-(1+\a)}\int_{\{|y|<\frac{|x|}2\}}|y|^{4\beta-N}\,{\rm d}y{\rm d}s\\
	&\le C|x|^{4\beta-N}t^{-(1+\a)}\int_0^{\frac t2}(1+r)^{-\gamma}\,{\rm d}r\le C|x|^{4\beta-N}\begin{cases}
	t^{-(\gamma+\a)},&\gamma<1,\\
	t^{-(1+\a)}\log t,&\gamma=1,\\
	t^{-(1+\a)},&\gamma>1,
	\end{cases}
	\end{aligned}
	\]
and hence
	\begin{equation*}\label{eq-bound II2}
		\|{\rm II}_2(\cdot,t)\|_{L^p(\{\nu <|x|/g(t)<\mu\})}\le C g(t)^{4\beta -N(1-\frac1p)}	\begin{cases}	
t^{-(\gamma+\a)},&\gamma<1,\\
t^{-(1+\a)}\log t,&\gamma=1,\\
t^{-(1+\a)},&\gamma>1.
	\end{cases}
	\end{equation*}	
	
Estimate \eqref{eq-bound intermediate} follows from the above bounds and the fact that $g(t)=o(t^{\a/(2\beta)})$.
	
To end the proof we have to check that \eqref{eq-bound intermediate} is sharp. To this aim we take $f(x,t)=(1+t)^{-\gamma}\chi_{B_1}(x)$.
Let $t$ be large enough so  that $g(t)>2/\nu$.  If $\nu g(t)<|x|<\mu g(t)$ and $|x-y|<1$,
\begin{equation}
\label{eq:optimality.intermediate}\frac\nu2 g(t)<\frac{|x|}2<|x|-\frac\nu2 g(t)<|x|-1<|y|<|x|+1<|x|+\frac\nu 2 g(t)<2|x|<2\mu g(t).
\end{equation}
Under these assumptions, if $s\in (0,(\nu g(t)/2)^{2\beta/\alpha})$, then $|y|s^{-\alpha/(2\beta)}\ge |y|/(\nu g(t)/2)\ge 1$. Therefore, if $g(t)=o(t^{\alpha/(2\beta)})$, using~\eqref{eq:Y.selfsimilar} and the estimates from below in~\eqref{eq-exterior estimates G1}--\eqref{eq-exterior estimates Gbeta}, and performing the change of variables $s=r g(t)^{2\beta/\alpha}$, we arrive at
\[\begin{aligned}
u(x,t)&\ge \int_0^{\big(\frac\nu2 g(t)\big)^{2\beta/\a}}\int_{\{|x-y|<1\}}(1+t-s)^{-\gamma}Y(y,s)\,{\rm d}y{\rm d}s\\
&\ge C t^{-\gamma} \int_0^{\big(\frac\nu2 g(t)\big)^{2\beta/\a}}s^{\a-1-\frac{\a N}{2\beta}}{\rm e}^{-c(g(t)s^{-\a/(2\beta)})^{\frac2{2-\a}}}\,{\rm d}s\\
&=C t^{-\gamma}g(t)^{\frac{2\beta}\a(\a-\frac{\a N}{2\beta})}\int_0^{\big(\frac\nu2\big)^{2\beta/\a}}r^{\a-1-\frac{\a N}{2\beta}}{\rm e}^{-cr^{-\frac\a{\beta(2-\a)}}}\,{\rm d}r=C t^{-\gamma}g(t)^{2\beta-N},
\end{aligned}
\]
Therefore,
\begin{equation}\label{eq-bound below intermediate 1}
	\|u(x,t)\|_{L^p(\{\nu <|x|/g(t)<\mu\})}\ge
	C t^{-\gamma} g(t)^{2\beta-N(1-\frac1p)}.
	\end{equation}
On the other hand, under the assumptions leading to \eqref{eq:optimality.intermediate},  if moreover $s\in(t/2,t)$ and  $t$ is large enough, we have $|y|<2\mu g(t)<(t/2)^{\a/(2\beta)}<s^{\a/(2\beta)}$.
Thus, using the estimate from below in~\eqref{eq-interior estimates G},
	\[\begin{aligned}
	u(x,t)&\ge C\int_{\frac t2}^t\int_{\{|x-y|<1\}}(1+t-s)^{-\gamma}s^{-(1+\a)} |y|^{4\beta-N}\,{\rm d}y{\rm d}s\\
	&\ge C t^{-(1+\a)}g(t)^{4\beta-N}\int_{\frac t2}^t(1+t-s)^{-\gamma}\,ds\\&=C t^{-(1+\a)}g(t)^{4\beta-N}\int_0^{\frac t2}(1+r)^{-\gamma}\,{\rm d}r\ge Cg(t)^{4\beta-N}\begin{cases}
	t^{-(\gamma+\a)},&\gamma<1,\\
	t^{-(1+\a)}\log t,&\gamma=1,\\
	t^{-(1+\a)},&\gamma>1.
	\end{cases}
	\end{aligned}\]
	Hence,
	\begin{equation}\label{eq-bound below intermediate 2}	\|u(\cdot,t)\|_{L^p(\{\nu <|x|/g(t)<\mu\})}\ge Cg(t)^{4\beta-N(1-\frac1p)}\begin{cases}
	t^{-(\gamma+\a)},&\gamma<1,\\
	t^{-(1+\a)}\log t,&\gamma=1,\\
	t^{-(1+\a)},&\gamma>1.
	\end{cases}
	\end{equation}
Estimates \eqref{eq-bound below intermediate 1}--\eqref{eq-bound below intermediate 2} show that \eqref{eq-bound intermediate} is sharp.
\end{proof}

 \section{Estimates in $\R^N$}\label{sect-conexion}
\setcounter{equation}{0}
In this section we establish the behavior of the \emph{global} $L^p(\R^N)$ norms of the mild solution to~\eqref{eq-f}, Theorem~\ref{teo-conexion}.

\begin{proof}[Proof of Theorem~\ref{teo-conexion}.]
	Due to  the results of theorems \ref{teo-exterior} and \ref{teo-compactos}, it is enough to show that the estimates are true in some region of the form $\{R\le |x|\le \delta t^{\a/(2\beta)}\}$ with  $R,\delta>0$.

We have $|u|\le{\rm I}+{\rm II}$, with ${\rm I}$ and ${\rm II}$ as in~\eqref{eq-decomposition intermediate1}. The term ${\rm I}$ is further decomposed as ${\rm I}={\rm I}_1+{\rm I}_2+{\rm I}_3$, with ${\rm I}_j$, $j\in\{1,2,3\}$ as in~\eqref{eq-decomposition intermediate2}.
Since $|x|<\delta t^{\a/(2\beta)}$ in the region we are interested in, taking  $\delta\in (0, 2^{1-\frac\a{2\beta}})$, then $(|x|/2)^{2\beta/\a}<t/2$. Therefore, reasoning as in Section~\ref{sect-intermediate},  we obtain~\eqref{eq:intermediate.I1}--\eqref{eq:intermediate.I2.I3}, from where
$$
\begin{aligned}
\|{\rm I}_j(\cdot,t)\|_{L^p(\{R<|x|<\delta t^{\a/(2\beta)}\})}&\le C
\begin{cases}
t^{\a-\gamma-\frac{\a N}{2\beta}(1-\frac1p)},&p\in [1,p_{\rm c}),\\
t^{-\gamma}\log t,	&p=p_{\rm c},\hskip2cm j\in\{1,2\},\\
t^{-\gamma},&p> p_{\rm c},
\end{cases}\\[0.3cm]
\|{\rm I}_3(\cdot,t)\|_{L^p(\{R<|x|<\delta t^{\a/(2\beta)}\})}
&\le C\begin{cases}
t^{\a-\gamma-\frac{\a N}{2\beta}(1-\frac1p)},&\gamma<1,\\
t^{\a-1-\frac{\a N}{2\beta}(1-\frac1p)}\log t,&\gamma=1,\hskip1cm p\in [1,p_*),\\
t^{\a-1-\frac{\a N}{2\beta}(1-\frac1p)},&\gamma>1,\\[0.3cm]
t^{-(\gamma+\a)}\log t,&\gamma<1,\\
t^{-(1+\a)}(\log t)^2,&\gamma=1,\hskip1cm p=p_*,\\
t^{-(1+\a)}\log t,&\gamma>1,\\[0.3cm]
t^{-(\gamma+\a)},&\gamma<1,\\
t^{-(1+\a)}\log t,&\gamma=1,\hskip1cm p>p_*,\\
t^{-(1+\a)},&\gamma>1.
\end{cases}
\end{aligned}
$$	
We conclude that	
\begin{equation*}
\label{eq-global bound I}
\|{\rm I}(\cdot,t)\|_{L^p(\{R<|x|<\delta t^{\a/2\beta}\})}\le C\begin{cases}
	t^{\a-\gamma-\frac{\a N}{2\beta}(1-\frac1p)},&\gamma<1,\\
	t^{\a-1-\frac{\a N}{2\beta}(1-\frac1p)}\log t,&\gamma=1,\hskip3.4cm p\in[1,p_{\rm c}),\\
	t^{\a-1-\frac{\a N}{2\beta}(1-\frac1p)},&\gamma>1,\\[0.3cm]	
	t^{-\gamma}\log t,&\gamma\le1,\\[-0.3cm]
&\hskip4.5cm p=p_{\rm c},\\[-0.3cm]
	t^{-1},&\gamma>1,\\[0.3cm]	
    t^{-\gamma}	&\gamma\le1-\a+\frac{\a N}{2\beta}(1-\frac1p),\\[-0.3cm]
&\hskip4.5cm p\in (p_{\rm c},p_*),\\[-0.3cm]
    t^{\a-1-\frac{\a N}{2\beta}(1-\frac1p)}	&\gamma\ge 1-\a+\frac{\a N}{2\beta}(1-\frac1p),\\[0.3cm]
	t^{-\gamma},&\gamma<1+\a,\\[-0.3cm]
&\hskip4.5cm p=p_*,\\[-0.3cm]
	t^{-(1+\a)}\log t,&\gamma\ge1+\a,\\[0.3cm]	
	t^{-\gamma},&\gamma<1+\a,\\[-0.3cm]
&\hskip4.5cm p>p_*,\\[-0.3cm]  	
	t^{-(1+\a)},&\gamma\ge1+\a.
	\end{cases}
	\end{equation*}
	
To analyze ${\rm II}$ we decompose it as ${\rm II}={\rm II}_1+{\rm II}_2$, where ${\rm II}_1$ and ${\rm II}_2$ are as in \eqref{eq-decomposition intermediate3}.	 We start with the subcritical case $p\in [1,p_{\rm c})$. Using~\eqref{eq-hypothesis f} and~\eqref{eq:p.norm.Y},
$$
\begin{aligned}
\|{\rm II}_1(\cdot,t)\|_{L^p(\{R<|x|<\delta t^{\a/(2\beta)}\})}
    &\le C \int_0^{\frac t2}(1+t-s)^{-\gamma}\|Y(\cdot,s)\|_{L^p(\R^N)}\,{\rm d}s\\
    &\le C t^{-\gamma}\int_0^{\frac t2}s^{\a-1-\frac{\a N}{2\beta}(1-\frac1p)}\,{\rm d}s= C t^{\a-\gamma-\frac{\a N}{2\beta}(1-\frac1p)},
\end{aligned}
$$

$$
\begin{aligned}
\|{\rm II}_2(\cdot,t)\|_{L^p(\{R<|x|<\delta t^{\a/(2\beta)}\})}
    &\le C \int_{\frac t2}^t(1+t-s)^{-\gamma}\|Y(\cdot,s)\|_{L^p(\R^N)}\,{\rm d}s\\
    &\le C t^{\a-1-\frac{\a N}{2\beta}(1-\frac1p)} \int_0^{\frac t2}(1+r)^{-\gamma}\,{\rm d}r\\
    &\le C
        \begin{cases}
		t^{\a-\gamma-\frac{\a N}{2\beta}(1-\frac1p)},&\gamma<1,\\
		t^{\a-1-\frac{\a N}{2\beta}(1-\frac1p)}\log t,&\gamma=1,\\
				t^{\a-1-\frac{\a N}{2\beta}(1-\frac1p)},&\gamma>1.				
		\end{cases}
\end{aligned}
$$

Let now $p\ge p_{\rm c}$. If $|y|<|x|/2$ and $|x|>R$, then $|x-y|\ge|x|/2\ge R/2$. Hence, taking $R$ large enough so that~\eqref{eq:decay.condition} holds outside $B_{R/2}$, we have that $|f(x-y,t-s)|\le C|x|^{-N}(1+t-s)^{-\gamma}$.
Hence, reasoning as in Section~\ref{sect-intermediate},
\[
{\rm II}_{1}(x,t)\le C|x|^{2\beta-N}t^{-\gamma},\qquad
{\rm II}_2(x,t)\le C|x|^{4\beta-N}\begin{cases}
	t^{-(\gamma+\a)},&\gamma<1,\\
	t^{-(1+\a)}\log t,&\gamma=1,\\
	t^{-(1+\a)},&\gamma>1,
	\end{cases}
\]	
and we get,
$$
\begin{aligned}
\|{\rm II}_{1}(\cdot,t)\|_{L^p(\{R<|x|<\delta t^{\a/(2\beta)}\})}&\le C \begin{cases}
	t^{-\gamma}\log t,&p=p_{\rm c},\\
	t^{-\gamma},&p>p_{\rm c},
	\end{cases}
\end{aligned}
$$
$$
\begin{aligned}	
		\|{\rm II}_2(\cdot,t)\|_{L^p(\{R<|x|<\delta t^{\a/(2\beta)}\})}&\le C	\begin{cases}	
t^{\a-\gamma-\frac{\a N}{2\beta}\big(1-\frac1p\big)},&\gamma<1,\\
t^{\a-1-\frac{\a N}{2\beta}\big(1-\frac1p\big)}\log t,&\gamma=1,\hskip1cm p\in[p_{\rm c},p_*),\\
t^{\a-1-\frac{\a N}{2\beta}\big(1-\frac1p\big)},&\gamma>1,\\[0.3cm]
	t^{-(\gamma+\a)}\log t,&\gamma<1,\\
	t^{-(1+\a)}(\log t)^2,&\gamma=1,\hskip1cm p=p_*,\\
	t^{-(1+\a)}\log t,&\gamma>1,\\[0.3cm]
	t^{-(\gamma+\a)},&\gamma<1,\\
	t^{-(1+\a)}\log t,&\gamma=1,\hskip1cm p>p_*,\\
	t^{-(1+\a)},&\gamma>1.
	\end{cases}
\end{aligned}
$$

The above estimates together with theorems~\ref{teo-exterior} and \ref{teo-compactos} yield the result.
\end{proof}		

\medskip

\noindent\textbf{Acknowledgments. } This project has received funding from the European Union's Horizon 2020 research and innovation programme under the Marie Sklodowska-Curie grant agreement No.\,777822. C.\,Cort\'azar supported by  FONDECYT grant 1190102 (Chile). F.\,Quir\'os supported by the Spanish Ministry of Economy, through  project  MTM2017-87596-P and under the ICMAT–Severo Ochoa grant CEX2019-000904-S.
	 N.\,Wolanski supported by
CONICET PIP625, Res. 960/12, ANPCyT PICT-2012-0153, UBACYT X117 and MathAmSud 13MATH03 (Argentina).

\end{document}